\documentclass[10pt,a4paper]{amsart}
\usepackage{amssymb}
 \usepackage[latin1]{inputenc}
 \usepackage[english]{babel}
 \usepackage{amsfonts}
 \usepackage{amsmath}
 \usepackage{amsthm}
 \usepackage[dvips]{graphicx}
 \DeclareGraphicsExtensions{.pdf,.png,.jpg,.eps}
 \usepackage{epstopdf}
 \usepackage{fancyhdr}
 \usepackage{colortbl}
 \usepackage[pdftex]{hyperref}
\numberwithin{equation}{section}
\newtheorem{theorem}{Theorem}[section]
\newtheorem{corollary}[theorem]{Corollary}

\theoremstyle{remark}
\newtheorem{remark}{Remark}[section]

\theoremstyle{definition}
\newtheorem{definition}{Definition}[section]

\newcommand{\R}{\mathbb{R}}
\newcommand{\C}{\mathbb{C}}
\newcommand{\Z}{\mathbb{Z}}
\newcommand{\N}{\mathbb{N}}

\begin{document}


\title[Magnetic Virial identities and applications to blow up]
{Magnetic Virial and applications to blow up for Schr\"odinger and wave equations}

\author{Andoni Garcia}
\address{Andoni Garcia: Universidad del Pais Vasco, Departamento de
Matem$\acute{\text{a}}$ticas, Apartado 644, 48080, Bilbao, Spain}
\email{andoni.garcia@ehu.es}

\begin{abstract}
We prove blow up results for the solution of the initial value problem with negative energy of the focusing mass-critical and supercritical nonlinear Schr\"odinger and the  focusing energy-subcritical nonlinear wave equations with electromagnetic potential. 

 \end{abstract}

\date{\today}


\maketitle
\section{Introduction}\label{sec:intronon}
In space dimension $n\geq 2$, we study the blow up for solutions with initial negative energy of the focusing $L^{2}$-critical and supercritical nonlinear Schr\"odinger equation with magnetic potential,
\begin{equation}\label{eq:nlsmag}
  \begin{cases}
    iu_{t}(t,x)-Hu(t,x)+|u|^{p-1}u(t,x)=0
    \\
    u(0,x)=f(x).
  \end{cases}
\end{equation}
Here the power of the nonlinearity corresponds to $1+\frac{4}{n}\leq p<1+\frac{4}{n-2}$. We also consider for $n\geq 3$ the blow up for solutions with initial negative energy of the focusing energy-subcritical nonlinear wave equation with magnetic potential,
\begin{equation}\label{eq:nlwmaggen}
  \begin{cases}
    u_{tt}(t,x)+Hu(t,x)-|u|^{p-1}u(t,x)=0
    \\
    u(0,x)=f(x)\\
    u_{t}(0,x)=g(x).
  \end{cases}
\end{equation}
For the wave equation we deal with the whole range of nonlinearities $1<p<1+\frac{4}{n-2}$. In both cases the solution is a complex valued function $u:\R^{1+n}\to\C$.

Therefore, instead of considering the free Schr\"odinger hamiltonian $H=-\Delta$, we treat with electromagnetic Schr\"odinger Hamiltonians in the standard covariant form
\begin{equation}
	H=-\nabla_{A}^{2}+V(x),
\end{equation}
where
\begin{equation}
\nabla_{A}=\nabla-iA,\qquad \Delta_{A}=\nabla_{A}^{2}.
\end{equation}
Here $A=(A^{1},\dots,A^{n}):\R^{n}\to\R^{n}$ is the magnetic potential and $V:\R^{n}\to\R$ is the electric potential. 

The study of the blow up phenomena for solutions of the free nonlinear Schr\"odinger equation goes back to Zakharov and Glassey (see \cite{Z} and \cite{Glass} respectively). It is based on a convexity method called \emph{virial identity}. For the free nonlinear wave equation the first result is due to Levine (see \cite{Lev}).

For the case of purely electric Schr\"odinger equation, i.e., for Schr\"odinger hamiltonians of the type $H=\Delta-V$, blow up results are given in \cite{C}. If we consider first order perturbations of the free Schr\"odinger hamiltonian by dealing with the nonlinear Schr\"odinger equation with magnetic potential, a blow up result for the solution of the Schr\"odinger equation when the magnetic potential $A$ is of the form
\begin{equation}\label{eq:GRA}
A=\frac{b}{2}(-y,x,0), 
\end{equation}
was obtained by Gon\c calves-Ribeiro (see \cite{GR}).

The aim of the paper is, in the case of the Schr\"odinger equation with magnetic potential, to generalize the previous example to any space dimension $n\geq2$ and to provide new examples  of potentials for which the solution of the focusing $L^2$-critical and supercritical nonlinear Schr\"odinger equation with magnetic potential \eqref{eq:nlsmag} blows up in finite time. Concerning the focusing energy-subcritical nonlinear wave equation with magnetic potential, namely \eqref{eq:nlwmaggen}, we will prove some blow up results, and we will conclude by showing that the magnetic potential $A$ has no influence in the blow up.

We will start by considering the setting of our problem. First, denote by $H^{1}_{A}(\R^{n})$ the following Hilbert space:
\begin{equation}
H^{1}_{A}=\left\{f:f\in L^{2}, \int|\nabla_{A}f|^{2}<\infty\right\}.
\end{equation}
Along through the paper we will assume some regularity on the Hamiltonian $H$. 

\begin{itemize}
\item[(\bf{H1})]
The Hamiltonian $H_{A}=-\nabla_{A}^{2}$ is essentially self-adjoint on $L^{2}(\R^{n})$, with form domain
\begin{equation*}
D(H_{A})=H^{1}_{A}=\left\{f:f\in L^{2}, \int|\nabla_{A}f|^{2}<\infty\right\}.
\end{equation*}
\item[(\bf{H2})] 
The potential $V$ is a perturbation of $H_A$ in the Kato-Rellich sense, i.e. there exists a small $\epsilon>0$ such that 
\begin{equation}
\|Vf\|_{L^2}\leq(1-\epsilon)\|H_Af\|_{L^2}+C\|f\|_L^2,
\end{equation}
for all $f\in D(H_A)$.
\item[(\bf{H3})]
The potentials $A$, $V$ are assumed to be $A\in\mathcal{C}^2$, $V\in\mathcal{C}^1$.
\end{itemize}
Assumptions (H1), (H2) have several consequences. First of all, they imply the self-adjointness of $H$, by standard perturbation techniques (see e.g. \cite{CFKS}); hence by the spectral theorem we can define the Sch\"odinger and wave propagators $S(t)=e^{itH}$, $W(t)=H^{-\frac{1}{2}}e^{it\sqrt{H}}$ and the powers $H^s$. Moreover, we can define for any $s$ the distorted norms
\begin{equation*}
\|f\|_{\dot{\mathcal{H}}^s}=\|H^{\frac{s}{2}}f\|_{L^2},\qquad \|f\|_{\mathcal{H}^s}=\|f\|_{L^2}+\|H^{\frac{s}{2}}f\|_{L^2}.
\end{equation*}

For the validity of (H1) and (H2) see the standard reference \cite{CFKS}. We do not make any attempt to optimize on the regularity assumptions that we need. Essentially the sufficient assumptions needed can be expressed in terms of the local integrability properties of the coefficients (see \cite{CFKS}, \cite{LS}). Clearly (H3) suffices for our purposes.

We can also consider the space $H^2_A(\R^n)$ by
\begin{equation*}
H^2_A=\{f:f\in L^2, H_Af\in L^2\}.
\end{equation*}
The corresponding norm will be given by:
\begin{equation*}
\|f\|_{H^2_A}=\left(\|f\|_{L^2}^2+\|H_Af\|_{L^2}^2\right)^{\frac{1}{2}}.
\end{equation*}

\begin{definition}For any $n\geq2$ the matrix-valued field $B:\R^{n}\to\mathcal{M}_{n\times n}(\R)$ is defined by 
\begin{equation*}
B:=DA-DA^{t},\qquad B_{ij}=\frac{\partial A^{i}}{\partial x^{j}}-\frac{\partial A^{j}}{\partial x^{i}}.
\end{equation*}
We also define the \emph{trapping component} of the vector field $B$ as  $B_{\tau}:\R^{n}\to\R^{n}$ given by:
\begin{equation*}
B_{\tau}=\frac{x}{|x|}B.
\end{equation*}
Hence $B$ is defined in terms of the anti-symmetric gradient of $A$. In dimension $n=3$, the previous definition identifies $B=\text{curl} \,A$, namely 
\begin{equation*}
Bv=\text{curl}\,A\wedge v,\quad\forall v\in\R^{3}.
\end{equation*}
In particular, we have 
\begin{equation}
B_{\tau}=\frac{x}{|x|}\wedge\text{curl}\,A,\quad n=3.
\end{equation}
Hence $B_{\tau}(x)$ is the projection of $B=\text{curl}\,A$ on the tangential space in $x$ to the sphere of radius $|x|$, for $n=3$. Observe also that $B_{\tau}\cdot x=0$ for any $n\geq2$, hence $B_{\tau}$ is a tangential vector field in any dimension. 
\end{definition}
The trapping component $B_{\tau}$ represents an obstruction to dispersion of solutions. Some explicit examples of magnetic potentials $A$ with $B_{\tau}=0$ in dimension 3 are given in \cite{F},\cite{FV}.\\
Moreover, by 
\begin{equation*}
\partial_{r}V=\nabla V\cdot\frac{x}{|x|},
\end{equation*}
we denote the radial derivative of $V$. We have to mention that, if the radial derivative is decomposed as 
\begin{equation*}
\partial_{r}V=(\partial_{r}V)_{+}-(\partial_{r}V)_{-},
\end{equation*}
the positive part $(\partial_{r}V)_{+}$ also represents an obstruction to dispersion. In \cite{FV}, $B_{\tau}$ and $(\partial_{r}V)_{+}$ are assumed to be small in a suitable sense in order to prove weak dispersive estimates. Also both components must be small in order to prove endpoint Strichartz estimates, as can be seen in \cite{DFVV}.

Our magnetic potential $A$ is assumed to satisfy the so called Coulomb gauge condition,
\begin{equation}
\text{div}\,A=0.
\end{equation}
Observe that this does not suppose a restriction, since $A$ and $A+\nabla\psi$ produce the same magnetic field $B$, for any $n\geq2$.

We now show some examples of magnetic potentials $A$ which satisfy our assumptions and which will be considered in the latter. The first kind of potentials are in some sense a natural generalization of the magnetic potential $A$ presented in \cite{C}, \cite{CE}, \cite{EL} and \cite{GR}.  Let us consider the $2\times2$ anti-symmetric matrix
\begin{equation*}
  \sigma
  :=
  \left(
  \begin{array}{cc}
    0 & -1
    \\
    1 & 0
  \end{array}\right).
\end{equation*}
For any even $n=2k\in\N$, we denote by $\Omega_n$ the $n\times n$ anti-symmetric matrix generated by $k$-diagonal blocks of $\sigma$, in the following way:
\begin{equation}\label{eq:Omega}
  \Omega_n
  :=
  \left(
  \begin{array}{cccc}
    \sigma & 0 & \cdots & 0
    \\
    0 & \sigma & \cdots & 0
    \\
    \vdots & \vdots & \ddots & \vdots
    \\
    0 & \cdots & 0 & \sigma
  \end{array}
  \right).
\end{equation}
In order to define the magnetic potential $A$ we will distinguish between odd and even dimension. If $n\geq3$ is an odd number, let us consider the following anti-symmetric matrix
 \begin{equation}\label{eq:Modd}
    M
    :=
    \left(
    \begin{array}{cc}
      \Omega_{n-1} & 0
      \\
      0 & 0
    \end{array}
    \right),
  \end{equation}
  where $\Omega_{n-1}$ is the $(n-1)\times(n-1)$-matrix defined in \eqref{eq:Omega}.

\noindent The analogous magnetic potential $A$ in even dimensions will be constructed in the following way. For $n=2$, consider the $2\times2$ anti-symmetric matrix 
\begin{equation}
M=\Omega_{2},
\end{equation}
and if $n\geq4$ is an even number, let $\mathbf{0}$ be the null $2\times2$-block and let us consider the following anti-symmetric $n\times n$-matrix
  \begin{equation}\label{eq:Meven}
    M
    :=
    \left(
    \begin{array}{cc}
      \Omega_{n-2} & 0
      \\
      0 & \mathbf{0}
    \end{array}
    \right),
  \end{equation}
  where $\Omega_{n-2}$ is the $(n-2)\times(n-2)$-matrix defined in \eqref{eq:Omega}.
  
\noindent Now, from the definition of $M$ we can define the following magnetic potential $A$,
  \begin{equation}\label{eq:AA}
    A(x)=\frac{1}{2}Mx,
  \end{equation}
where the expression of $M$ depends on the dimension $n\geq2$. 
\begin{remark}
This kind of magnetic potentials $A$ generated from an anti-symmetric matrix $M$ were also shown in \cite{FG}. More precisely if we consider the magnetic potential $A$ of the form
\begin{equation}
 A=|x|^{-\alpha}Mx,\quad1<\alpha<2,   
\end{equation}
Strichartz estimates for the linear version of \eqref{eq:nlsmag} with $V=0$ are false in the whole range of Schr\"odinger admissibility. These counterexamples for the magnetic Schr\"odinger equation were clearly inspired by the ones produced in the electric case in the paper \cite{GVV}. More concretely, the counterexamples in \cite{GVV} are based on potentials of the form 
\begin{equation}
V(x)=(1+|x|^2)^{-\frac{\alpha}{2}}\omega\left(\frac{x}{|x|}\right),\qquad 0<\alpha<2,
\end{equation}
where $\omega$ is a positive scalar function, homogeneous of degree $0$, which has a non degenerate minimum point $P\in\ S^{n-1}$. Moreover, it is crucial there to assume that $\omega(P)=0$. The main idea is to approximate $H=-\Delta+V(x)$ by a second order Taylor expansion, with an harmonic oscillator. Then, the condition $\alpha<2$ causes the lack of global (in time) dispersion.
\end{remark}

\noindent The other magnetic potentials $A$ that we are going to consider must satisfy that the trapping part of the magnetic field $B$ must be identically zero
\begin{equation*}
B_{\tau}=0.
\end{equation*}
When $n=3$, some examples are the ones appearing in \cite{FV}. These are described as follows. First we consider singular potentials. Take
\begin{equation}
A=\frac{1}{x^{2}+y^{2}+z^{2}}(-y,x,0)=\frac{1}{x^{2}+y^{2}+z^{2}}(x,y,z)\wedge(0,0,1).
\end{equation}
We can check that 
\begin{equation*}
\nabla\cdot A=0,\quad B=-2\frac{z}{(x^{2}+y^{2}+z^{2})^{2}}(x,y,z),\quad B_{\tau}=0.
\end{equation*}
Another (more singular potential) example is the following:
\begin{equation}
A=\left(\frac{-y}{x^{2}+y^{2}},\frac{x}{x^{2}+y^{2}},0\right)=\frac{1}{x^{2}+y^{2}}(x,y,z)\wedge(0,0,1).
\end{equation}
Here, we have $B=(0,0,\delta)$, with $\delta$ denoting Dirac's delta function. Again we have $B_{\tau}=0$.\\
The electric potential $V$ will be assumed to be
\begin{equation}
V\in L^{r}(\R^{n})+L^{\infty}(\R^{n}),
\end{equation}
for some $r>n/2$.\\
Notice that there is a quantity associated to the solutions of the equations \eqref{eq:nlsmag} and \eqref{eq:nlwmaggen}. We define the \emph{energy} for the nonlinear magnetic Schr\"odinger equation as
\begin{equation}\label{eq:energy}
E_{S}(t)=\frac{1}{2}\int_{\R^{n}}|\nabla_{A}u(x)|^{2}\,dx+\frac{1}{2}\int_{\R^{n}}V(x)|u(x)|^{2}\,dx-\frac{1}{p+1}\int_{\R^{n}}|u(x)|^{p+1}\,dx.
\end{equation}
The energy associated to the nonlinear magnetic wave equation is the following one
\begin{equation}\label{eq:energyW}
E_{W}(t)=\frac{1}{2}\int_{\R^{n}}\left(|u_{t}(x)|^{2}+|\nabla_{A}u(x)|^{2}\,dx+V(x)|u(x)|^{2}\right)\,dx-\frac{1}{p+1}\int_{\R^{n}}|u(x)|^{p+1}\,dx.
\end{equation}
It is well known that if we consider the equation \eqref{eq:nlsmag} with initial datum $f\in H^{1}_{A}(\R^{n})$ and the equation \eqref{eq:nlwmaggen} with initial data $(f,g)\in(H^{1}_{A}(\R^{n})\times L^{2}(\R^{n}))$, these initial value problems are \emph{energy sub-critical}. Therefore, there exists a unique local solution of \eqref{eq:nlsmag} and \eqref{eq:nlwmaggen} defined in a maximal time of existence. For the existence theorems we refer the Section \ref{sec:localexistence} below. 

\noindent Let us now study the Cauchy problem \eqref{eq:nlsmag} for a given initial datum $f\in\Sigma$, where $\Sigma$ is the space defined by 
\begin{equation}
\Sigma:=\{f\in H^{1}_{A}(\R^{n}): |x|f\in L^{2}(\R^{n})\}.
\end{equation}
In the Schr\"odinger case we will prove that under some assumptions made on the potentials, the unique local solution $u\in\mathcal{C}((-T_{*},T^{*});H^{1}_{A}(\R^{n}))$ of the Cauchy problem \eqref{eq:nlsmag} in the maximal time of existence $(-T_{*},T^{*})$ blows up in finite time.

This will be based on a convexity method called \emph{virial identity}. In the setting of the free nonlinear Schr\"odinger equation is due to Zakharov and Glassey (see \cite{Z}, \cite{Glass}).  When instead of considering the free hamiltonian, we deal with the electromagnetic hamiltonian $H$, the virial identity differs from the one performed in the free case. This can be seen in \cite{FV}, where virial identities for the linear version of the equations \eqref{eq:nlsmag} and \eqref{eq:nlwmaggen} are presented. A $3$-D version of the virial identity for the magnetic Schr\"odinger equation appears in \cite{GR1}, \cite{GR}. We give the statement of the theorems for the Schr\"odinger and wave equations and the proofs will be performed in Section \ref{sec:blowup}.

For the wave equation, we will proceed following Levine (see \cite{Lev}).

Now we are in condition to state the theorems for the blow up for the solution of the focusing $L^2$-critical and supercritical nonlinear Schr\"odinger equation with magnetic potential and the focusing energy-subcritical nonlinear wave equation with magnetic potential. The hypotheses concerning the local existence are included in the statements (see Section 2). The statement for the Schr\"odinger case is the following.
\begin{theorem}\label{thm:Sblowup}
Let $n\geq2$ and consider the Cauchy problem \eqref{eq:nlsmag}
for $1+\frac{4}{n}\leq p<1+\frac{4}{n-2}$.
Let $f\in\Sigma$ such that $E_{S}(0)<0$ Assume (H1), (H2) and that $A$ is of the form \eqref{eq:AA} or $A$ and $V$ satisfy the assumptions of Theorem \ref{thm:StriScro3} and $B\tau=0$. Moreover, assume that
\begin{itemize}
\item[(i)] $V+\frac{1}{2}rV_{r}\geq0$.
\end{itemize}
Then, the unique solution $u\in\mathcal{C}((-T_{*},T^{*});\Sigma)$ of the focusing equation \eqref{eq:nlsmag} blows up in finite time.
\end{theorem}
For the wave equation, we are able to prove the following result.
\begin{theorem}\label{thm:Wblowup}
Let $n\geq 3$ and consider the Cauchy problem \eqref{eq:nlwmaggen}
for $1<p<1+\frac{4}{n-2}$. Here $f\in H^{1}_{A}(\R^{n})$, $g\in L^{2}(\R^{n})$. Let $E_{W}(0)<0$. Assume (H1), (H2) and that $A$ and $V$ satisfy the assumptions of Theorem \ref{thm:StriWav3} and further that 
\begin{itemize}
 \item[(i)]$V\geq0$.
\end{itemize}
Then, the unique solution $u\in\mathcal{C}([0,T^{*});H^{1}_{A}(\R^{n}))$, $u_{t}\in\mathcal{C}([0,T^{*});L^{2}(\R^{n}))$ of the focusing equation \eqref{eq:nlwmaggen} blows up in finite time.
\end{theorem}

The rest of the paper is organized as follows. In Section \ref{sec:localexistence} we state and give the proofs of the theorems for the local existence of solution for the equations \eqref{eq:nlsmag} and \eqref{eq:nlwmaggen}. Section \ref{sec:magneticvirial} is devoted to the virial identities for Schr\"odinger and wave equations and finally, in Section \ref{sec:blowup} we give the proofs of the main theorems of the paper, Theorems 
\ref{thm:Sblowup} and \ref{thm:Wblowup}.

\section{Local existence results}\label{sec:localexistence}

In this section we announce the theorems for the local existence of solution for the equations \eqref{eq:nlsmag} and \eqref{eq:nlwmaggen}. Therefore, we consider the initial value problem for the focusing nonlinear Schr\"odinger equation with magnetic potential \eqref{eq:nlsmag}. The nonlinearity will be assumed to be energy-subcritical, namely $1<p<1+\frac{4}{n-2}$.
Concerning the wave equation we deal with the initial value problem for the focusing energy-subcritical nonlinear wave equation with magnetic potential \eqref{eq:nlwmaggen}.
It can be proved that, for the Schr\"odinger equation, whenever the nonlinearity $p$ is energy subcritical, then for $f\in H^{1}_{A}$, there exists a unique local solution of the Cauchy problem \eqref{eq:nlsmag}. The proof of the result relies on the Strichartz estimates that are satisfied by the solution $u$ of the linear version of the equation \eqref{eq:nlsmag}. We will distinguish between the two kind of magnetic potentials $A$ that we presented above. For the magnetic potentials $A$ of the form \eqref{eq:AA} the homogeneous and inhomogeneous Strichartz estimates appear in \cite{CE} for the case $n=3$ and these estimates can be generalized in the same way to the case $n\geq2$. For the second kind of magnetic potentials, namely the ones for which $B_\tau=0$, the Strichartz estimates that we use depend on the dimension $n=2$ and $n\geq3$. In 2D we can apply the homogeneous Strichartz estimates from \cite{DF} and derive the nonhomogeneous ones by an standard $TT^{*}$-argument. The potentials $A$ and $V$ should be small in order the estimates to be valid. For $n\geq3$ the estimates considered are the ones appearing in \cite{DFVV}. We refer to the case $n\geq3$ by including the precise estimates.
\begin{definition}
Let $n\geq3$. A measurable function $V(x)$ is said to be in the \textit{Kato class} $K_n$ provided
\begin{equation*}
\lim_{r\downarrow0}\sup_{x\in\R^n}\int_{|x-y|\leq r}\frac{|V(y)|}{|x-y|^{n-2}}\,dy=0.
\end{equation*}
We shall usually omit the reference to the space dimension and write $K$ instead of $K_n$. The \textit{Kato norm} is defined as
\begin{equation*}
\|V\|_K=\sup_{x\in\R^n}\int_{|x-y|\leq r}\frac{|V(y)|}{|x-y|^{n-2}}\,dy.
\end{equation*}
A last notation we shall need is the radial-tangential norm
\begin{equation*}
\|f\|_{L^p_rL^{\infty}(Sr)}^p:=\int_0^\infty\sup_{|x|=r}|f|^p\,dr.
\end{equation*}
\end{definition}
The result is the following (see \cite{DFVV} Theorem 1.1)
\begin{theorem}\label{thm:StriScro3}
Let $n\geq3$. Given $A,V\in C^1_{loc}(\R^n\setminus\{0\})$, assume (H1), (H2). Moreover assume that 
\begin{equation}
\|V_{-}\|_K<\frac{\pi^{\frac{n}{2}}}{\Gamma(\frac{n}{2}-1)}
\end{equation}
and 
\begin{equation}
\sum_{j\in\Z}2^j\sup_{x\in C_j}|A|+\sum_{j\in\Z}2^{2j}\sup_{x\in C_j}|V|<\infty,
\end{equation}
where $C_j=\{x: 2^j\leq |x|\leq2^{j+1}\}$ and the Coulomb gauge condition
\begin{equation}
\nabla\cdot A=0.
\end{equation}
Finally, when $n=3$, we assume that for some $M>0$
\begin{equation}
\frac{(M+\frac{1}{2})^2}{M}\||x|^{\frac{3}{2}}B_\tau\|_{L^{\infty}}^2+(2M+1)\||x|^2(\partial_rV)_{+}\|_{L^1_rL^{\infty}(S_r)}<\frac{1}{2},
\end{equation}
while for $n\geq4$ we assume that 
\begin{equation}
\||x|^2B_{\tau}(x)\|_{L^{\infty}}^2+2\||x|^3(\partial_rV)_{+}(x)\|_{L^{\infty}}<\frac{2}{3}(n-1)(n-3).
\end{equation}
Then, for any Schr\"odinger admissible couple (p,q), the following Strichartz estimates hold:
\begin{equation}\label{eq:strichhomsch}
\|e^{itH}\varphi\|_{L^pL^q}\leq C\|\varphi\|_{L^2}, \qquad \frac{2}{p}=\frac{n}{2}-\frac{n}{q},\qquad p\geq2,\qquad p\neq2\ \text{if}\ n=3.
\end{equation}
In dimension $n=3$, we have the endpoint estimate
\begin{equation}\label{eq:stricinhomsch}
\||D|^\frac{1}{2}e^{itH}\varphi\|_{L^2L^6}\lesssim\|H^\frac{1}{4}\varphi\|_{L^2}.
\end{equation}
\end{theorem}

\begin{remark}\label{rm:endpoint}
By a standard TT*-argument, for admissible couples $(p,q),(\tilde p', \tilde q')$ as in \eqref{eq:strichhomsch} it follows the inhomogeneous estimate
\begin{equation}\label{eq:inomog}
  \left\| \int_0^t e^{i(t-s)H}F(s)\,ds \right\|_{L^pL^q}
\lesssim
\|F\|_{L^{\tilde p'}L^{\tilde q'}}.
\end{equation}
On the other hand, we cannot argue the above estimate by \eqref{eq:stricinhomsch} in the 3D-endpoint case $(p,q)=(2,6)$. By the way, estimates \eqref{eq:strichhomsch} and \eqref{eq:inomog} are enough to prove the following local well-posedness result.
\end{remark}

Now, we give the statement of the local existence result for the Schr\"odinger equation in dimension $n\geq3$.

\begin{theorem}Let  A be of the form \eqref{eq:AA} or A and V satisfy the assumptions of Theorem \ref{thm:StriScro3}.
Let $1<p<1+\frac{4}{n-2}$. Given $f\in H^{1}_{A}(\R^{n})$, then there exists a unique maximal solution $u\in\mathcal{C}((-T_{*},T^{*});H^{1}_{A}(\R^{n}))$ of \eqref{eq:nlsmag} such that $u(0)=f$. It holds that if $T_{*}<\infty$ or $T^{*}<\infty$, then $\|u\|_{H^{1}_{A}}\to\infty$ as $t\downarrow -T_{*}$ or $\|u\|_{H^{1}_{A}}\to\infty$ as $t\uparrow T^{*}$ respectively. In addition, we have that conservation of the mass and the energy hold, that is for all $t\in(-T_{*},T^{*})$,
\begin{equation}
\int_{\R^{n}}|u(t,x)|^{2}\,dx=\int_{\R^{n}}|f(x)|^{2}\,dx
\end{equation}
\begin{equation}
E_{S}(t)=E_{S}(0).
\end{equation}
Moreover, if $f\in H^2_A(\R^n)$, then $u\in \mathcal{C}((-T_{*},T^{*});H^{2}_{A}(\R^{n}))\cap\mathcal{C}^1((-T_{*},T^{*});L^2(\R^{n}))$.

\end{theorem}
\begin{proof}
The proof of the local existence goes straightforward just applying the corresponding Strichartz estimates depending on the magnetic potential $A$ we are considering. The conservation of the mass follows by multiplying the equation \eqref{eq:nlsmag} by $\bar{u}$, integrating by parts and taking the resulting imaginary part. The conservation of the energy comes from multiplying the equation \eqref{eq:nlsmag} by $\bar{u_t}$, integrating by parts and taking the resulting real part. All the computations are justified for solutions $u\in H^2_A(\R^n)$. Then, proceed by a density argument in order to obtain the result for $u\in H^1_A(\R^n)$. The proof of the persistence of the solution in $H^2_A(\R^n)$ for a given $f\in H^2_A(\R^n)$ can be found for $n=3$ in \cite{CE}.
\end{proof}
\begin{remark}
We also want to mention the existence result appearing in \cite{CE} for the $3$D case. They consider a constant magnetic field that, without loss of generality, can be assumed to be 
\begin{equation}
B=(0,0,b),
\end{equation}
for some $b\in\R-\{0\}$. Therefore, up to a gauge transform, the magnetic potential $A$ can be chosen in the following way
\begin{equation}
A=\frac{b}{2}(-y,x,0).
\end{equation}
\end{remark}

\begin{remark}Notice that in \cite{EL}, general nonlinearities of the type $g(x,u)$ are considered. The version we present here is the corresponding to power-type nonlinearities.
\end{remark}

For the case of $A\equiv0$, the initial value problem for the nonlinear Schr\"odinger equation has been studied in the past by, among other, Ginibre and Velo (See \cite{GV1}, \cite{GV2}, \cite{GV3}) , Kato \cite{K} or Cazenave and Weissler (See \cite{CW1}, \cite{CW2}) . The methods are of a perturbative nature and rely basically on sharp dispersive properties of the linear equation.

Concerning the nonlinear magnetic wave equation \eqref{eq:nlwmaggen}, a local existence result can be proved. As in the Schr\"odinger case, the result is based in the Strichartz estimates for the solution $u$ of the linear version of the equation \eqref{eq:nlwmaggen}. The estimates that we are going to consider appear in \cite{FV}. We include the statement of the theorem.

\begin{theorem}\label{thm:StriWav3}
Let $n\geq3$. Assume (H1), (H2) and either
\begin{equation}
\||x|^3B\|_{L^1_rL^{\infty}(S_r)}+\||x|^2(\partial_rV)_{+}\|_{L^1_rL^{\infty}(S_r)}\leq\frac{1}{2},
\end{equation}
for $n=3$ or
\begin{equation}
|B_\tau(x)|\leq\frac{C_1}{|x|^2},\qquad |(\partial_rV)_{+}\|\leq\frac{C_2}{|x|^3}, \qquad C_1^2+2C_2\leq\frac{2}{3}(n-1)(n-3),
\end{equation}
for $n\geq4$. Moreover assume that
\begin{equation}
|B(x)|\leq\frac{C}{(1+|x|)^{2+\delta}},\qquad |V(x)|\leq\frac{C}{(1+|x|)^{2+\delta}},
\end{equation}
for some $C>0$ and some $\delta>0$. Then, for any non-endpoint admissible couple (p,q), the following Strichartz estimates hold:
\begin{equation}
\|u\|_{L^p\dot{H}^\sigma_q}\lesssim\|f\|_{\dot{\mathcal{H}}^1}+\|g\|_{L^2},
\end{equation}
\begin{equation}
\frac{2}{p}+\frac{n-1}{q}=\frac{n-1}{2}, \qquad2\leq p\leq\infty, \qquad \frac{2(n-1)}{n-3}\geq q\geq 2,\qquad q\neq\infty.
\end{equation}
\end{theorem}

\begin{remark}
We consider the analogous of Remark \ref{rm:endpoint} for the case of the wave equation.
\end{remark}

Relying on the last result, the following theorem for the local existence can be proved.

\begin{theorem} Let $A$ and $V$ satisfy the assumptions of Theorem \ref{thm:StriWav3}.
Let $1<p<1+\frac{4}{n-2}$. Given $f\in H^{1}_{A}(\R^{n})$, $g\in L^{2}(\R^{n})$, then there exists a unique maximal solution $u\in\mathcal{C}([0,T^{*});H^{1}_{A}(\R^{n}))$, $u_{t}\in\mathcal{C}([0,T^{*});L^{2}(\R^{n}))$ of \eqref{eq:nlwmaggen} such that $u(0)=f$, $u_{t}(0)=g$. In addition, we have that conservation of the energy holds, that is for all $t\in[0,T^{*})$,
\begin{equation}
E_{W}(t)=E_{W}(0).
\end{equation}
\end{theorem}

\begin{proof}
The proof of the local existence follows by a fixed point argument just applying the Strichartz estimates given above (see \cite{FV}, Theorem 1.13). The conservation of the energy follows by multiplying the equation \eqref{eq:nlwmaggen} by $\bar{u_t}$, integrating by parts and taking the real part. All of the computations are justified for solutions $u\in H^2_A(\R^n)$. Then by a density argument, we obtain the result for solutions $u\in H^1_A(\R^n)$.
\end{proof}

In addition, for $L^{2}$-subcritical nonlinearities i.e., $1<p<1+\frac{4}{n}$,  the solution of the Schr\"odinger equation \eqref{eq:nlsmag} is in fact global in time. The theorem reads as follows.

\begin{theorem}
Let $1<p<1+\frac{4}{n}$ and $f\in H^{1}_{A}(\R^{n})$. Let $u\in\mathcal{C}((-T_{*},T^{*});H^{1}_{A}(\R^{n}))$ be a local solution of \eqref{eq:nlsmag}, and maximal time of existence. Then $T_{*}=T^{*}=\infty$. 
\end{theorem}
\begin{proof}
It can be derived directly from the conservation of the energy.
\end{proof}

\section{Magnetic virial identities}\label{sec:magneticvirial}

In this section we present the virial identities for the magnetic Schr\"odinger equation and the magnetic wave equation in any dimension. For the free nonlinear Schr\"odinger equation the identity is due to Zakharov \cite{Z} and Glassey \cite{Glass}. When the free equation is perturbed by a electromagnetic potential, Fanelly and Vega (see \cite{FV}) performed the corresponding virial identities for the linear Sch\"odinger and linear wave equations. For the Schr\"odinger equation in 3-D and magnetic potential $A$ of the form \eqref{eq:GRA} a virial identity appears in  \cite{GR1}, \cite{GR}. Theorem \ref{thm:virial} gives the virial identity for Schr\"odinger, while in Theorem \ref{thm:virialwave} appears the corresponding one for the wave equation. The identities are the same ones of \cite{FV}, just adding the term corresponding to the nonlinear contribution. We have to say that in despite of for the wave equation we do not use the virial identity in order to prove the blow up of the solution, we will include it for the sake of completeness.

\begin{theorem}[Virial for magnetic nonlinear Schr\"odinger] \label{thm:virial}Let $\phi:\R^{n}\to\R$ be a radial, real-valued multiplier, $\phi=\phi(|x|)$, and let
\begin{equation}
	\Theta_{S}(t)=\int_{\R^{n}}\phi|u|^{2}\,dx.
\end{equation}
Then, for any solution u of the magnetic nonlinear Schr\"odinger equation \eqref{eq:nlsmag} with initial datum $f\in L^{2}$, $H_{A}f\in L^{2}$, the following virial-type identities hold:
\begin{equation}\label{eq:virial1}
	\dot{\Theta}_{S}(t)=2\Im\int_{\R^{n}}\bar{u}(t,x)\nabla_{A}u(t,x)\cdot \nabla\phi(x)\,dx.
\end{equation}
\begin{align}\label{eq:virial2}
	\ddot{\Theta}_{S}(t)=&4\int_{\R^{n}}\nabla_{A}uD^{2}\phi\overline{\nabla_{A}u}\,dx-\int_{\R^{n}}|u|^{2}\Delta^{2}\phi\,dx
	-2\int_{\R^{n}}\phi'V_{r}|u|^{2}\,dx\\
	&+4\Im\int_{\R^{n}}u\phi'B_{\tau}\cdot\overline{\nabla_{A}u}\,dx-2\frac{(p-1)}{(p+1)}\int_{\R^{n}}|u|^{p+1}\Delta\phi\,dx.
	\nonumber
\end{align}
where
\begin{equation*}
	(D^{2}\phi)_{jk}=\frac{\partial^{2}}{\partial x^{j}\partial x^{k}}\phi,\qquad\Delta^{2}\phi=\Delta(\Delta\phi),
\end{equation*}
for $j,k=1,\dots,n$, are respectively the Hessian matrix and the bi-Laplacian of $\phi$.

\end{theorem}

\begin{proof}
The proof of \eqref{eq:virial2} comes from the one performed in \cite{FV} for the linear version of the equation \eqref{eq:nlsmag}. More precisely, instead of considering the electric potential $V$ we just apply the results in \cite{FV} for the potential created adding the nonlinear contribution, namely 
\begin{equation}
\tilde{V}=V-|u|^{p-1}. 
\end{equation}
Therefore, just recalling that for the linear equation,
\begin{equation}
\dot{\Theta}_{S}(t)=2\Im\int \bar{u}\nabla_{A}u\cdot\nabla\phi\,dx,
\end{equation}
clearly this term is not affected by the nonlinear part. Now, if we consider the second derivative of $\Theta_S$, for the linear version of the equation we got
\begin{align}\label{eq:viriallinear}
	\ddot{\Theta}_{S}(t)=&4\int_{\R^{n}}\nabla_{A}uD^{2}\phi\overline{\nabla_{A}u}\,dx-\int_{\R^{n}}|u|^{2}\Delta^{2}\phi\,dx\\
	&-2\int_{\R^{n}}\phi'V_{r}|u|^{2}\,dx
	+4\Im\int_{\R^{n}}u\phi'B_{\tau}\cdot\overline{\nabla_{A}u}\,dx.
	\nonumber
\end{align}
\begin{remark}\label{rm:regularity}
As is pointed out in \cite{FV}, in the computations which lead us to \eqref{eq:viriallinear} the highest order term in $u$ that appears is of the form
\begin{equation*}
\int\nabla^2_Au\nabla\phi\cdot\overline{\nabla_{A}u};
\end{equation*}
it makes sense since $f\in L^2$, $H_Af\in L^2$, which implies $H_Ae^{itH}f\in L^2$, and by interpolation $\nabla_Ae^{itH}f\in L^2$.
\end{remark}
The main difference appears in the term involving $V$. We include the details in order to be clear. We proceed as follows just considering the potential $\tilde{V}$. Since $V\in\mathcal{C}^1$, it holds
\begin{equation}\label{eq:derVtilde}
-\int\phi'\tilde{V}_{r}|u|^{2}=-\int\nabla\phi\cdot\nabla \tilde{V}|u|^{2}=-\int\phi'V_{r}|u|^{2}+\int\nabla\phi\cdot\nabla(|u|^{p-1})|u|^{2}.
\end{equation}
Integrating by parts, noting that the unique solution $u$ of \eqref{eq:nlsmag} stays in $H^1_A$ for all time, we obtain
\begin{align}\label{eq:nolin1}
\int\nabla\phi\cdot\nabla(|u|^{p-1})|u|^{2}=&-\int\Delta\phi|u|^{p+1}-\int|u|^{p-1}\nabla\phi\cdot\nabla(|u|^2)
\nonumber
\\
=&-\int\Delta\phi|u|^{p+1}-2\Re\int|u|^{p-1}\nabla\phi\cdot\nabla u\bar{u}
\nonumber
\\
=&-\int\Delta\phi|u|^{p+1}-\frac{2}{p+1}\int\nabla\phi\cdot\nabla(|u|^{p+1})
\nonumber
\\
=&-\frac{p-1}{p+1}\int\Delta\phi|u|^{p+1}.
\end{align}
\begin{remark}Notice that all the computations above are justified since $u\in H^1_A(\R^n)$, just applying Sobolev embedding.
\end{remark}
The result follows directly from \eqref{eq:viriallinear}, \eqref{eq:derVtilde} and \eqref{eq:nolin1}.

\end{proof}

\begin{theorem}[Virial for magnetic nonlinear wave]\label{thm:virialwave} Let $\phi,\Psi:\R^{n}\to\R$, be two radial, real-valued multipliers, and let
\begin{equation}
\Theta_{W}(t)=\int_{\R^{n}}\left(\phi|u_{t}|^{2}+\phi|\nabla_{A}u|^{2}-\frac{1}{2}(\Delta\phi)|u|^{2}\right)\,dx+\int_{\R^{n}}|u|^{2}\phi V\,dx+\int_{\R^{n}}|u|^{2}\Psi\,dx.
\end{equation}
Then, for any solution $u$ of the magnetic nonlinear wave equation \eqref{eq:nlwmaggen} with initial data $f,g\in L^{2}$, $H_{A}f,H_{A}g\in L^{2}$, the following virial-type identity holds:
\begin{align}\label{eq:virial3}
\ddot{\Theta}_{W}(t)=&2\int_{\R^{n}}\nabla_{A}uD^{2}\phi\overline{\nabla_{A}u}\,dx-\frac{1}{2}\int_{\R^{n}}|u|^{2}\Delta^{2}\phi\,dx\\
&+2\int_{\R^{n}}|u_{t}|^{2}\Psi\,dx-2\int_{\R^{n}}|\nabla_{A}u|^{2}\Psi\,dx+\int_{\R^{n}}|u|^{2}\Delta\Psi\,dx
\nonumber
\\
&-\int_{\R^{n}}\phi' V_{r}|u|^{2}\,dx-2\int_{\R^{n}}\Psi V|u|^{2}\,dx+2\Im\int_{R^{n}}u\phi' B_{\tau}\cdot\overline{\nabla_{A}u}\,dx
\nonumber
\\
&+\int_{\R^{n}}\left(2\Psi-\frac{p-1}{p+1}\Delta\phi\right)|u|^{p+1}\,dx.
\nonumber
\end{align}

\end{theorem}

\begin{proof}
We can argue as in the case of the Schr\"odinger equation just considering the analogous to Remark \ref{rm:regularity}. Recall the expression for the second derivative of $\Theta_W$ for the linear version of \eqref{eq:nlwmaggen} that appears in \cite{FV}. It reads as follows.
\begin{align}
\ddot{\Theta}_{W}(t)=&2\int_{\R^{n}}\nabla_{A}uD^{2}\phi\overline{\nabla_{A}u}\,dx-\frac{1}{2}\int_{\R^{n}}|u|^{2}\Delta^{2}\phi\,dx\\
&+2\int_{\R^{n}}|u_{t}|^{2}\Psi\,dx-2\int_{\R^{n}}|\nabla_{A}u|^{2}\Psi\,dx+\int_{\R^{n}}|u|^{2}\Delta\Psi\,dx
\nonumber
\\
&-\int_{\R^{n}}\phi' V_{r}|u|^{2}\,dx-2\int_{\R^{n}}\Psi V|u|^{2}\,dx+2\Im\int_{R^{n}}u\phi' B_{\tau}\cdot\overline{\nabla_{A}u}\,dx.
\nonumber
\end{align}
Now, by considering the potential
\begin{equation}
\tilde{V}=V-|u|^{p-1},
\end{equation}
and proceeding as in the Schr\"odinger case for the nonlinear part, the result follows easily.

\end{proof}

\begin{remark}Notice that $B_{\tau}$ only involves the terms with $\phi$. This will be useful in Section \ref{sec:blowup}.
\end{remark}
We give two corollaries of the previous theorems.
\begin{corollary}Let u be a solution of the magnetic nonlinear Schr\"odinger equation \eqref{eq:nlsmag}  with $f\in L^{2}$, $H_{A}f\in L^{2}$. Then the variance
\begin{equation*}
	Q(t)=\int_{\R^{n}}|x|^{2}|u|^{2}\,dx
\end{equation*}
satisfies the identities
\begin{equation}
	\dot{Q}(t)=4\Im\int_{\R^{n}}\bar{u}(t,x)\nabla_{A}u(t,x)\cdot x\,dx.
\end{equation}
\begin{align}\label{eq:virialmul}
	\ddot{Q}(t)=&8\int_{\R^{n}}|\nabla_{A}u|^{2}\,dx
	-4\int_{\R^{n}}|x|V_{r}|u|^{2}\,dx\\
	&+8\Im\int_{\R^{n}}|x|uB_{\tau}\cdot\overline{\nabla_{A}u}\,dx-4\frac{n(p-1)}{(p+1)}\int_{\R^{n}}|u|^{p+1}\,dx.
	\nonumber
\end{align}
\end{corollary}

\begin{corollary}Let u be a solution of the magnetic nonlinear wave equation \eqref{eq:nlwmaggen} with $f, g\in L^{2}$, $H_{A}f, H_{A}g\in L^{2}$. Then the quantity
\begin{equation*}
Q(t)=\int_{\R^{n}}\left\{|x|^{2}\left(|u_{t}|^{2}+|\nabla_{A}u|^{2}+|u|^{2}V\right)-(n-1)|u|^{2}\right\}\,dx
\end{equation*}
satisfies the identity
\begin{align}
\ddot{Q}(t)=&2\int_{\R^{n}}|u_{t}|^{2}+|\nabla_{A}u|^{2}\,dx-2\int_{\R^{n}}|x|V_{r}|u|^{2}\,dx-2\int_{\R^{n}}V|u|^{2}\,dx
\nonumber
\\
&+4\Im\int_{\R^{n}}|x|uB_{\tau}\cdot\overline{\nabla_{A}u}\,dx+2\left(1-n\frac{p-1}{p+1}\right)\int_{\R^{n}}|u|^{p+1}\,dx.
\end{align}
\end{corollary}
The proofs of the corollaries are immediate applications of identities \eqref{eq:virial2} and \eqref{eq:virial3} with the choice $\phi=|x|^{2}$, $\Psi\equiv1$.

Let us now consider the virial identity \eqref{eq:virialmul} particularized to the case of $L^{2}$-critical or supercritical nonlinearity, namely $1+\frac{4}{n}\leq p<1+\frac{4}{n-2}$. It reads as follows,
\begin{align}\label{eq:virialmulcrit}
	\ddot{Q}(t)=&8\int_{\R^{n}}|\nabla_{A}u|^{2}\,dx
	-4\int_{\R^{n}}|x|V_{r}|u|^{2}\,dx\\
	&+8\Im\int_{\R^{n}}|x|uB_{\tau}\cdot\overline{\nabla_{A}u}\,dx-4\frac{n(p-1)}{(p+1)}\int_{\R^{n}}|u|^{p+1}\,dx.
	\nonumber
\end{align}
The last quantity can be expressed in terms of the energy \eqref{eq:energy}. Hence
\begin{align}\label{eq:virialener}
	\ddot{Q}(t)=&16E_{S}(t)-8\int_{\R^{n}}V|u|^{2}\,dx\\
	&-4\int_{\R^{n}}|x|V_{r}|u|^{2}\,dx
	+8\Im\int_{\R^{n}}|x|uB_{\tau}\cdot\overline{\nabla_{A}u}\,dx
	\nonumber
	\\
	&+\frac{16-4n(p-1)}{p+1}\int_{\R^{n}}|u|^{p+1}\,dx,
	\nonumber
\end{align}
and since the energy is conserved, $E_{S}(t)=E_{S}(0)$ for all times $t\in(-T_*,T^*)$, it holds
\begin{align}\label{eq:virialenerini}
	\ddot{Q}(t)=&16E_{S}(0)-8\int_{\R^{n}}V|u|^{2}\,dx\\
	&-4\int_{\R^{n}}|x|V_{r}|u|^{2}\,dx
	+8\Im\int_{\R^{n}}|x|uB_{\tau}\cdot\overline{\nabla_{A}u}\,dx
	\nonumber
	\\
	&+\frac{16-4n(p-1)}{p+1}\int_{\R^{n}}|u|^{p+1}\,dx.
	\nonumber
\end{align}
Now, since we have $1+\frac{4}{n}\leq p<1+\frac{4}{n-2}$ the last term in \eqref{eq:virialenerini} is clearly negative. Therefore, it can be expected that, given $f\in\Sigma$ such that the corresponding energy $E_{S}(0)<0$, then the solution $u(t)$ of \eqref{eq:nlsmag} blows up in finite time. 

\section{Proofs of theorems \ref{thm:Sblowup} and \ref{thm:Wblowup}}\label{sec:blowup}
In this section we give the proofs of the Theorems \ref{thm:Sblowup} and \ref{thm:Wblowup}. We will prove that, under some assumptions for magnetic potential $A$ and the electric potential $V$, the local solution of the Cauchy problems \eqref{eq:nlsmag} and \eqref{eq:nlwmaggen} with initial data of negative energy actually blows up in finite time. For the $3$-D Schr\"odinger equation, a blow up result  is given by Gon\c calves-Ribeiro in \cite{GR}. In some sense, our result includes a natural generalization to higher dimensions. During the proof we will use the virial identity given by \eqref{eq:virialmul}, with mass-critical and supercritical nonlinearity, written in terms of the energy, namely \eqref{eq:virialener}.

\begin{proof}[Proof of Theorem \ref{thm:Sblowup}]We will argue by contradiction showing that there exists a finite time $T^+$ satisfying $Q(T^+)<0$. This contradicts the fact that $Q$ is nonnegative for all times. Recall that if we start with an initial datum $f\in\Sigma$ the solution $u$ is a $\Sigma$-valued continuous function. The proof can be found for the case $n=3$ in \cite{GR} (see the third step of the proof of the Theorem $1.2$). Let $f\in\Sigma$ such that $E_{S}(0)<0$. Let us assume $(i)$.

\noindent Now, recalling \eqref{eq:virialener}, we get that 
\begin{align}\label{eq:virialtype1}
		\ddot{Q}(t)=&16E_{S}(t)-8\int_{\R^{n}}V|u|^{2}\,dx\\
	&-4\int_{\R^{n}}|x|V_{r}|u|^{2}\,dx
	+8\Im\int_{\R^{n}}|x|uB_{\tau}\cdot\overline{\nabla_{A}u}\,dx
	\nonumber
	\\
	&+\frac{16-4n(p-1)}{p+1}\int_{\R^{n}}|u|^{p+1}\,dx.
	\nonumber
\end{align}
Let us now expand the first term of the previous equality. We have that
\begin{align}\label{eq:expand1}
\int_{\R^{n}}|\nabla_{A}u|^{2}\,dx=&\int_{\R^{n}}|\nabla u|^{2}\,dx+\int_{\R^{n}}|A|^{2}|u|^{2}\,dx-2\Re\int_{\R^{n}}\nabla\bar{u}\cdot iAu\,dx\\
\nonumber
=&\int_{\R^{n}}|\nabla u|^{2}\,dx+\int_{\R^{n}}|A|^{2}|u|^{2}\,dx+2\Im\int_{\R^{n}}\nabla\bar{u}\cdot Au\,dx.
\end{align}
Notice that for the the magnetic potentials $A$ given \eqref{eq:AA}, it holds
\begin{equation}
-2A(x)=|x|B_{\tau}(x).
\end{equation}
Therefore, we obtain
\begin{equation}\label{eq:expand2}
\Im\int_{\R^{n}}|x|uB_{\tau}\cdot\overline{\nabla_{A}u}\,dx=-2\Im\int_{\R^{n}}\nabla\bar{u}\cdot Au\,dx-2\int_{\R^{n}}|A|^{2}|u|^{2}\,dx.
\end{equation}
Hence, from \eqref{eq:virialtype1} particularized to the case $1+\frac{4}{n}\leq p<1+\frac{4}{n-2}$, \eqref{eq:expand1} and \eqref{eq:expand2}, we get
\begin{align}\label{eq:conii}
\ddot{Q}(t)\leq&8\int_{\R^{n}}|\nabla u|^{2}\,dx-8\int_{\R^{n}}|A|^{2}|u|^{2}\,dx
\nonumber
\\
&+8\int_{\R^{n}}V|u|^{2}\,dx-\frac{16}{(p+1)}\int_{\R^{n}}|u|^{p+1}\,dx\leq CE_{S}(0),
\end{align}
for some positive constant $C$.

If we consider that $A$ is such that $B_\tau=0$, we have that trivially holds
\begin{equation}\label{eq:coniii}
	\ddot{Q}(t)\leq16E_{S}(0).
\end{equation}
Integrating relations \eqref{eq:conii} or \eqref{eq:coniii} with respect to time twice, we obtain,
\begin{equation}
	Q(t)\leq CE_{S}(0)t^{2}+4t\Im\int_{\R^{n}}\bar{u}(0,x)\nabla_{A}u(0,x)\cdot x\,dx+Q(0),
\end{equation}
for some positive constant $C$. Now, since $Q(0)$ is finite and $E_{S}(0)<0$, there would be a finite time $T^{+}$ such that $Q(T^{+})<0$. This however is in contradiction with the fact that, by definition $Q(t)\geq0$ for all times. This proves that the solution stops to exist in a finite time.

\end{proof}

Now, we will proof blow up for the solution of the focusing energy-subcritical nonlinear magnetic wave equation \eqref{eq:nlwmaggen}. It will be shown that the magnetic potential $A$ has no influence in the blow up phenomena, that is, it seems that only the electric potential $V$ plays a role.

\begin{proof}[Proof of Theorem \ref{thm:Wblowup}]We will proceed following the proof performed by Levine for the free nonlinear wave equation (see \cite{Lev}). Similar arguments have been performed in \cite{Ball}, \cite{Glass1}, \cite{PS} and \cite{Tsu} among others. Let us consider
\begin{equation}
F(t)=\|u\|_{L^{2}}^{2}=\langle u,u\rangle.
\end{equation}
We will show that if $f$ and $g$ are chosen correctly, then $F(t)$ goes to infinity in finite time. Therefore, suppose that we can find $\alpha>0$ and initial data $f$ and $g$ such that
\begin{itemize}
\item[(a)]$(F(t)^{-\alpha})''\leq0$,\quad$\forall t\geq0$,
\item[(b)]$(F(t)^{-\alpha})'<0$,\quad$t=0$.
\end{itemize}
Then $F(t)^{-\alpha}$ will go to zero in finite time. This will imply that $F(t)$ will go to infinity in finite time.

\noindent We have that, taking the derivative of $F^{-\alpha}$, it holds
\begin{equation}
(F^{-\alpha})'=-\alpha F^{-\alpha-1}F',
\end{equation}
and
\begin{equation}
(F^{-\alpha})''=-\alpha F^{-\alpha-2}[FF''-(\alpha+1)(F')^{2}].
\end{equation}
Hence, whenever $F(t)\neq0$ if we could show that $FF''-(\alpha+1)(F')^{2}\geq0$, it would follow $(a)$, and therefore $F^{-\alpha}$ would be concave. Thus, if $F(0)\neq0$, we will have for all t for which $u(t)$ exists,
\begin{equation}\label{eq:ineqF}
F^{-\alpha}(t)\leq F^{-\alpha}(0)-\alpha tF'(0)F^{-\alpha-1}(0),
\end{equation}
since the graph of a concave function must lie below any tangent line. The condition $(b)$ will follow if $f$ and $g$ have the same sign, since
\begin{equation}
(F^{-\alpha})'(0)=-\alpha F^{-\alpha-1}(0)F'(0)=-2\alpha F(0)^{-\alpha-1}\Re\langle u_{t}(0),u(0)\rangle.
\end{equation}
Hence, \eqref{eq:ineqF} is equivalent to have
\begin{equation}
F^{\alpha}(t)\geq F^{\alpha+1}(0)[F(0)-\alpha tF'(0)]^{-1},
\end{equation}
and therefore as $t\to T(\leq F(0)/\alpha F'(0))$, from below (if $F'(0)>0$), we see that $F(t)\to\infty$.
Therefore, we have to show that 
\begin{equation}
FF''-(\alpha+1)(F')^{2}\geq0.
\end{equation}
Let us start by computing the derivatives of $F$. We obtain
\begin{equation}
F'(t)=2\Re\langle u_{t},u\rangle,
\end{equation}
and
\begin{equation}
F''(t)=2\langle u_{t},u_{t}\rangle+2\Re\langle u_{tt},u\rangle.
\end{equation}
Therefore if we denote by $Q(t)=FF''-(\alpha+1)(F')^{2}$, we have
\begin{align}
Q(t)=&\left\{4(\alpha+1)\langle u_{t},u_{t}\rangle+2(\Re\langle u_{tt},u\rangle-(2\alpha+1)\langle u_{t},u_{t}\rangle)\right\}F
\\
&-(\alpha+1)\left(2\Re\langle u_{t},u\rangle\right)^{2}
\nonumber
\\
&=4(\alpha+1)\left\{\left(\int|u_{t}|^{2}\,dx\right)\left(\int|u|^{2}\,dx\right)-\left(\Re\int u_{t}\bar{u}\,dx\right)^{2}\right\}
\nonumber
\\
&+2F\left\{\Re\int u_{tt}\bar{u}\,dx-(2\alpha+1)\int|u_{t}|^{2}\,dx\right\}.
\nonumber
\end{align}
The first term on the right is positive by Cauchy-Schwarz, so we need only to show that $H(t)\geq0$, where
\begin{equation}\label{eq:H}
H(t)=\Re\int u_{tt}\bar{u}\,dx-(2\alpha+1)\int|u_{t}|^{2}\,dx.
\end{equation}
From \eqref{eq:nlwmaggen}, we have that
\begin{equation}
u_{tt}=\nabla_{A}^{2}u-Vu+|u|^{p-1}u,
\end{equation}
then , by substituting this in \eqref{eq:H} we arrive to
\begin{align}
H(t)=&\Re\int u_{tt}\bar{u}\,dx-(2\alpha+1)\int|u_{t}|^{2}\,dx
\\
=&\Re\left\{\int\nabla_{A}^{2}u\bar{u}\,dx-\int V|u|^{2}\,dx+\int|u|^{p+1}\,dx\right\}-(2\alpha+1)\int|u_{t}|^{2}\,dx
\nonumber
\\
=&-\int|\nabla_{A}u|^{2}\,dx-\int V|u|^{2}\,dx+\int|u|^{p+1}\,dx-(2\alpha+1)\int|u_{t}|^{2}\,dx.
\end{align}
Thus, if we choose $\alpha$ so that
\begin{equation}
2(2\alpha+1)=p+1,
\end{equation}
we have
\begin{align}
H(t)=-(p+1)E_{W}(t)+2\alpha\int\left(|\nabla_{A}u|^{2}+V|u|^{2}\right)\,dx.
\end{align}
Now, by $(i)$ and the conservation of the energy, we have that 
\begin{equation}
H(t)\geq-(p+1)E_{W}(t)=-(p+1)E_{W}(0).
\end{equation}
Now, suppose that $E_{W}(0)<0$. Since $\alpha=\frac{p-1}{4}>0$, then $H(t)\geq0$ and hence $(F(t)^{-\alpha})''\leq0$. Also, we can choose $f$ and $g$ such that 
\begin{equation}
(F^{-\alpha})'(0)=-\alpha F^{-\alpha-1}(0)F'(0)=-2\alpha F(0)^{-\alpha-1}\Re\langle u_{t}(0),u(0)\rangle<0.
\end{equation}
For such initial data, $F(t)$ goes to infinity in finite time.

\end{proof}

\end{document}